\documentclass[10pt]{article}

\usepackage{amsthm,amsmath,amssymb,tikz}

\usepackage{graphicx}

\usepackage[colorlinks=true,citecolor=blue,linkcolor=blue,urlcolor=blue]{hyperref}

\newcommand{\arxiv}[1]{\href{http://arxiv.org/abs/#1}{\texttt{arXiv:#1}}}

\theoremstyle{plain}
\newtheorem{theorem}{Theorem}[section]

\newtheorem{corollary}[theorem]{Corollary}
\newtheorem{proposition}[theorem]{Proposition}

\theoremstyle{definition}
\newtheorem{definition}[theorem]{Definition}
\newtheorem{example}[theorem]{Example}

\newtheorem{problem}[theorem]{Problem}
\newtheorem{question}[theorem]{Question}

\theoremstyle{remark}
\newtheorem{remark}[theorem]{Remark}

\title{\bf Quasisymmetric functions for nestohedra}

\author{Vladimir Gruji\'c\\
\small Faculty of Mathematics\\[-0.8ex]
\small Belgrade University, Serbia\\[-0.8ex]}

\date{
}

\begin{document}

\maketitle
\begin{abstract}
For a generalized permutohedron $Q$ the enumerator $F(Q)$ of
positive lattice points in interiors of maximal cones of the
normal fan $\Sigma_Q$ is a quasisymmetric function. We describe
this function for the class of nestohedra as a Hopf algebra
morphism from a combinatorial Hopf algebra of building sets. For
the class of graph-associahedra the corresponding quasisymmetric
function is a new isomorphism invariant of graphs. The obtained
invariant is quite natural as it is the generating function of
ordered colorings of graphs and satisfies the recurrence relation
with respect to deletions of vertices.
\end{abstract}

\section{Introduction}
Let $Q$ be a convex polytope. The {\it normal fan} $\Sigma_Q$ is
the set of cones over the faces of the polar polytope $Q^\ast$.
The polytope $Q$ is simple if and only if the normal fan
$\Sigma_Q$ is simplicial. The polytope $Q$ is a {\it Delzant}
polytope if its normal fan $\Sigma_Q$ is regular, i.e. the
generators of the normal cone $\sigma_v$ at any vertex $v\in Q$
can be chosen to form an integer basis of $\mathbb{Z}^n$.

The permutohedron $Pe^{n-1}$ is a $(n-1)$-dimensional polytope
which is the convex hull $Pe^{n-1}={\rm Conv}\{x_\omega\mid
\omega\in S_n\}$, where $x\in\mathbb{R}^n$ is a point with
strictly increasing coordinates $x_1<\cdots<x_n$ and
$x_\omega=(x_{\omega(1)},\ldots,x_{\omega(n)})$ for a permutation
$\omega\in S_n$. The normal fan $\Sigma_{Pe^{n-1}}$ of the
permutohedron $Pe^{n-1}$ is the braid arrangement fan. Postnikov
introduced in \cite{P} a class of convex polytopes called
generalized permutohedra, which includes some interesting
subclasses with a rich combinatorial structure, such as matroid
base polytopes, graphic zonotopes, nestohedra and
graph-associahedra.

A {\it generalized permutohedron} $Q$ in $\mathbb{R}^{n}$ is a
convex polytope characterized by equivalent conditions (see
\cite[Theorem 15.3]{PRW} for the general statement)
\begin{itemize}
\item[(i)] the normal fan $\Sigma_Q$ is refined by the braid
arrangement fan $\Sigma_{Pe^{n-1}}$

\item[(ii)] any edge lies in a direction $e_i-e_j$ for some $1\leq
i,j\leq n$

\item[(iii)] $Q$ is a Minkowski summand of the permutohedron
$Pe^{n-1}$.
\end{itemize}
We regard a function $f:[n]\rightarrow\mathbb{N}$ as an element of
$(\mathbb{R}^{n})^{\ast}$ by $\langle
f,x\rangle=\sum_{i=1}^{n}f(i)x_i, x\in\mathbb{R}^{n}$. For a
generalized permutohedron $Q$ in $\mathbb{R}^n$ a function
$f:[n]\rightarrow\mathbb{N}$ is called $Q$-{\it generic} if it has
a unique maximum over  $Q$ at a vertex $\displaystyle{\max_{x\in
Q}}\langle f,x\rangle=\langle f,v\rangle$. Thus it lies in the
interior of the normal cone $\sigma_v$ for some vertex $v\in Q$.
Let $F(Q)$ be the enumerator function of $Q$-generic functions

$$F(Q)=\sum_{f: \ Q-\mbox{\rm\small generic}}\mathbf{x}_f=
\sum_{v\in Q}\sum_{f\in\sigma_v}\mathbf{x}_f,$$ where
$\mathbf{x}_f=x_{f(1)}\cdots x_{f(n)}$. This power series is
introduced and its main properties are derived by Billera, Jia and
Reiner in (\cite{BJR}, Section 9). It is a homogeneous
quasisymmetric function of degree $n$. Consider its expansion in
the monomial basis of quasisymmetric functions

$$F(Q)=\sum_{\alpha\models n}\zeta_\alpha(Q)M_\alpha,$$
where $M_\alpha=\sum_{i_1<\cdots<i_k}x_{i_1}^{a_1}\cdots
x_{i_k}^{a_k}$ for a composition $\alpha=(a_1,\ldots,a_k)\models
n$ of the integer $n$.

If $Q=Z_\Gamma$ is a graphic zonotope the function $F(Z_\Gamma)$
is easily seen to be Stanley's chromatic symmetric function
$X_\Gamma$ of the graph $\Gamma$ \cite{S1}. For the matroid base
polytope $Q=P_M$ the quasisymmetric function $F(P_M)$ is an
isomorphism invariant of a matroid $M$ introduced by Billera, Jia
and Reiner in \cite{BJR}. The unifying principle of these two
examples is a construction of certain combinatorial Hopf algebras
such that prescribed invariants are obtained by the universal
morphism to quasisymmetric functions. The theory of combinatorial
Hopf algebras is developed by Aguiar, Bergeron and Sotille in
\cite{ABS}. We particularly respond to \cite[Problem 9.3]{BJR} and
study the quasisymmetric functions $F(Q)$ for the class of
nestohedra.


The nestohedron $Q=P_B$ is a simple polytope obtained from a
simplex by a sequence of face truncations. The family of faces by
which truncations are performed is encoded by a building set $B$,
which is a subset of the face lattice of the simplex. The ground
sets of connected subgraphs of a graph $\Gamma$ produce the
graphical building set $B(\Gamma)$. The class of polytopes
$P_{B(\Gamma)}$ is called graph-associahedra. It contains an
important series of polytopes such as associahedra or Stasheff
polytopes, cyclohedra or Bott-Taubes polytopes, stellohedra and
permutohedra. For the class of nestohedra we describe coefficients
$\zeta_\alpha(P_B)$ in terms of underlying building set $B$. We
construct a certain combinatorial Hopf algebra of building sets
$\mathcal{B}$ and show that the canonical morphism maps a building
set $B$ precisely to the generating function $F(P_B)$ of the
corresponding nestohedron $P_B$.


Recently, some of quasisymmetric refinements of Stanley's
chromatic symmetric function are appeared, see \cite{H},
\cite{SW}. We introduce a new quasisymmetric function invariant
$F_\Gamma$ associated to a graph $\Gamma$ which has independent
combinatorial and algebraic descriptions as

1) the enumerator function of $P_{B(\Gamma)}$-generic functions,

2) the Hopf morphism from certain combinatorial Hopf algebra of
graphs,

3) the enumerator function of {\it ordered colorings} of $\Gamma$.

\noindent We say a coloring of a graph is ordered if colors are
linearly ordered and monochromatic vertices are not connected by
paths trough vertices colored by smaller colors. In addition the
function $F_\Gamma$ satisfies the recurrence relation with respect
to deletions of vertices
$$F_\Gamma=\sum_{v\in V}(F_{\Gamma\setminus v})_1,$$
where $F\mapsto (F)_1$ is a certain shifting operator on
quasisymmetric functions.

The paper is organized as follows. In section 2 we review the
necessary facts about nestohedra. In section 3 we review weak
orders and preorders and their connections with combinatorics of
the permutohedron. In section 4 we construct the combinatorial
Hopf algebra $\mathcal{B}$ and prove that the assignment $B\mapsto
F(P_B)$ comes from the universal Hopf algebra morphism to
quasisymmetric functions. In section 5 the function $F(P_B)$ is
related with the multiset of unlabeled rooted trees associated to
vertices of $P_B$. In section 6 the theory of $P$-partitions is
used to determine the action of the antipode of quasisymmetric
functions on $F(P_B)$. In section 7 we give a graph theoretic
interpretation of the invariant $F(P_{B(\Gamma)})$. We prove the
recurrence relation for $F_\Gamma$ with respect to deletions of
vertices of a graph which serves as the main computational tool.
As an application we compute $F(Q)$ for $Q$ be a permutohedron,
associahedron, cyclohedron or stellohedron.  As the conclusion
some open problems concerning the graph invariant $F_\Gamma$ and
Hopf algebra $\mathcal{B}$ are posed.

\section{Nestohedra}

In this section we review the necessary definitions and facts
about nestohedra. This class of polytopes is introduced and
studied in \cite{FS}, \cite{P}, \cite{PRW}, \cite{Z}.

Let $\Delta_{[n]}={\rm Conv}\{e_1,\ldots,e_n\}$ be the standard
coordinate simplex in $\mathbb{R}^n$. To a subset $I\subset [n]$
corresponds the face $\Delta_I={\rm Conv}\{e_i\mid i\in
I\}\subset\Delta_{[n]}$. A hypergraph $B$ on the finite set
$[n]=\{1,\ldots,n\}$ is a collection of nonempty subsets of $[n]$.
For convenience we suppose that $\{i\}\in B, i\in [n]$. For a
hypergraph $B$ on $[n]$ define the polytope $P_B$ as the Minkowski
sum of simplicies

\[P_B=\sum_{I\in B}\Delta_I=\sum_{I\in B}{\rm Conv}\{e_i\mid i\in
I\}={\rm Conv}\sum_{I\in B}\{e_i\mid i\in I\}.\] The polytope
$P_B$ is simple if additionally the hypergraph $B$ satisfies the
following condition:

$\diamond$ If $I, J\in B$ and $I\cap J\neq\emptyset$ then $I\cup
J\in B$.

\noindent In that case $B$ is called a {\it building set} and the
polytope $P_B$ is called a {\it nestohedron}.

\begin{example}\label{graph}
Given a simple graph $\Gamma$ on the vertex set $[n]$, the
graphical building set $B(\Gamma)$ is defined as the collection of
all $I\subset [n]$ such that the induced subgraphs $\Gamma\mid_I$
are connected. Carr and Devadoss studied polytopes $P_{B(\Gamma)}$
in \cite{CD} and called them {\it graph-associahedra}. For
instance the series $Pe^{n-1}, As^{n-1}, Cy^{n-1}, St^{n-1}, n>2$
of permutohedra, associahedra, cyclohedra and stellohedra
correspond respectively to complete graphs $K_n$, path graphs
$L_n$, cycle graphs $C_n$ and star graphs $K_{1,n-1}$ on $n$
vertices.
\end{example}

Let $B_{\rm max}$ be the collection of maximal by inclusion
elements of a building set $B$. We say that a building set $B$ is
{\it connected} if $[n]\in B$. Since the Minkowski sum is the
product for polytopes which are contained in the complementary
subspaces, we have

\[P_B=\sum_{I\in B_{\rm
max}}\sum_{J\in B\mid_I}\Delta_J=\prod_{I\in B_{\rm
max}}P_{B\mid_I}.\] Thus we may restrict ourselves to connected
building sets. The realization of nestohedra is given by the
following proposition.

\begin{proposition}[\cite{FS}, Proposition 3.12]\label{equalities}
Let $B$ be a connected building set on the finite set $[n]$ and
$\mu(B)$ be the number of elements of $B$. The nestohedron $P_B$
can be described as the intersection of the hyperplane $H_{[n]}$
with the halfspaces $H_{I,\geq}$ corresponding to all $I\in
B\setminus\{[n]\}$, where
\[H_{[n]}=\{x\in\mathbb{R}^n\mid
\sum_{i\in[n]}x_i=\mu(B)\},\]
\[H_{I,\geq}=\{x\in\mathbb{R}^n\mid \sum_{i\in
I}x_i\geq \mu(B\mid_I)\}.\]
\end{proposition}

As a consequence we have that the nestohedron $P_B$ corresponding
to a connected building set $B$ is realized from the dilated
standard coordinate simplex
$\{(x_1,\ldots,x_n)\in\mathbb{R}^{n}\mid x_1+\cdots+x_n=\mu(B),
x_1,\ldots,x_n\geq 1\}$ by a sequence of face truncations which
are encoded by elements $I\in B\setminus\{[n]\}$. For instance the
truncations along all faces of the simplex
$\{(x_1,\ldots,x_n)\in\mathbb{R}^{n}\mid x_1+\cdots+x_n=2^{n}-1,
x_1,\ldots,x_n\geq 1\}$ realize the permutohedron $Pe^{n-1}$.

The face lattice of $P_B$ is described by the following
proposition.

\begin{proposition}[\cite{FS}, Theorem 3.14; \cite{P}, Theorem 7.4]\label{nested}
Given a connected bu\-ilding set $B$ on $[n]$, let $\{F_I\mid I\in
B \setminus\{[n]\}\}$ be the set of facets of the nestohedron
$P_B$. The intersection $F_{I_1}\cap\ldots\cap F_{I_k}, k\geq2$ is
a nonempty face of $P_B$ if and only if
\begin{itemize}
\item[({\rm N1})] $I_i\subset I_j$ or $I_j\subset I_i$ or $I_i\cap
I_j=\emptyset$ for any $1\leq i<j\leq k$. \item[({\rm N2})]
$I_{j_1}\cup\cdots\cup I_{j_p}\notin B$ for any pairwise disjoint
sets $I_{j_1},\ldots,I_{j_p}$.
\end{itemize}
\end{proposition}
A subcollection $\{I_1,\ldots,I_k\}\subset B$ that satisfies the
conditions (N1) and (N2) is called a {\it nested set}. The
collection $N_B$ of all nested sets form a simplicial complex
called the {\it nested set complex}. The face poset of $N_B$ is
opposite to the face poset of $P_B$. Therefore $N_B$ may be
realized as a simplicial polytope which is polar to $P_B$.

Proposition \ref{nested} implies that vertices of $P_B$ correspond
to maximal nested sets. We denote this correspondence by $v\mapsto
N_v$. To a vertex $v\in P_B$ associate the poset
$(N_v\cup\{[n]\},\subset)$. For $I\in N_v\cup\{[n]\}$ let $i_I\in
[n]$ be the element such that $\{i_I\}=I\setminus\cup\{J\in
N_v\mid J\subsetneq I\}$. The correspondence $I\mapsto i_I$ is a
well defined bijection by the characterization of maximal nested
sets (\cite{P}, Proposition 7.6). It defines the partial order
$\leq_v$ on $[n]$ by $i_I\leq_v i_J$ if and only if $I\subset J$
in $N_v\cup\{[n]\}$. Denote this poset on $[n]$ by $P_v$. The
Hasse diagram $T_v$ of the poset $P_v$ for $v\in P_B$ is called a
$B$-{\it tree} \cite[Definition 8.1]{PRW}. So $(i,j)\in T_v$ if
and only if $i\lessdot_{P_v}j$ is a covering relation in the poset
$P_v$. The root of $T_v$ is the maximal element of $P_v$.


The following proposition, which is a consequence of Proposition
\ref{equalities}, describes the coordinates and normal cones at
vertices of $P_B$. Note that any nested set
$\{I_1,\ldots,I_k\}\subset B$ is ordered by inclusion of sets. The
usual covering relations is denoted by $J\lessdot I$.

\begin{proposition}\label{cones}
Let $v\in P_B$ be a vertex of the nestohedron $P_B$ and $N_v\in
N_B$ be the corresponding maximal nested set.
\begin{itemize}
\item[(i)] The coordinates of the vertex $v$ are given by
\[x_{i_I}=\mu(B\mid_I)-\sum_{J\in N_v: J\lessdot I}\mu(B\mid_J), I\in
N_v\cup\{[n]\}.\]

\item[(ii)] The interior of the normal cone $\sigma_v$ at the
vertex $v$ is determined by the inequalities \[x_i<x_j, \
\mbox{for all} \ (i,j)\in T_v.\]
\end{itemize}
\end{proposition}

\section{Preorders, Weak orders and Permutohedra}

A binary relation $\precsim$ is called a {\it preorder} on the
finite set $V=\{v_1,\ldots,v_n\}$ if it is reflexive and
transitive. If it is in addition total, i.e. $u\precsim v$ or
$v\precsim u$ for all $u,v\in V$, the preorder $\precsim$ is
called a {\it weak order} or a total preorder. The preorder
defines an equivalence relation by $u\thicksim v$ if and only if
$u\precsim v$ and $v\precsim u$. The relation $\precsim/\thicksim$
is a partial order on the set of equivalence classes
$V/\thicksim$. If $\precsim$ is a weak order on $V$ then
$\precsim/\thicksim$ is a total order on $V/\thicksim$. Any weak
order is represented as an ordered partition of $V$, i.e. as the
ordered family $(V_1,\ldots,V_k)$ of nonempty disjoint subsets
which covers $V$. The relation is recovered by $u\precsim v$ if
and only if $u\in V_i$ and $v\in V_j$ for some $1\leq i\leq j\leq
k$. The type of a weak order $\precsim$ is the corresponding
composition $\mbox{type}(\precsim)=(|V_1|,\ldots,|V_k|)\models n$
and $k$ is its length. Any function $f:V\rightarrow\mathbb{N}$
determines a weak order on $V$ by $u\precsim_f v$ if $f(u)\leq
f(v)$, for all $u,v\in V$. For any strictly increasing function
$g:\mathbb{N}\rightarrow\mathbb{N}$ we have
$\precsim_f=\precsim_{g\circ f}$. To a weak order $\precsim$ on
$V$ is associated the monomial quasisymmetric function

$$M_{{\rm type}(\precsim)}=\sum_{\precsim_f=\precsim}\mathbf{x}_f.$$

Let $\mathbf{WO}(n)=\cup_{k=1}^{n}\mathbf{WO}_k(n)$ be the set of
all weak orders of the set $V$ graded by the lengths. To an
ordered partition $(V_1,\ldots,V_k)$ is associated the flag
$\emptyset\subset V_1\subset V_1\cup V_2\subset\ldots\subset
V_1\cup\ldots\cup V_{k-1}\subset V$. This is a one-to-one
correspondence between ordered partitions and flags on $V$.
Therefore the set of all weak orders $\mathbf{WO}(n)$ is modelled
as the simplicial complex $\Delta[n]^{(1)}$ the first barycentric
subdivision of the simplex on $V$. The simplicial complex
$\Delta[n]^{(1)}$ is combinatorially equivalent to the convex
simplicial polytope whose polar polytope is the permutohedron
$Pe^{n-1}$ (see \cite{O}). Thus $k$-faces of $Pe^{n-1}$ are
labeled by ordered partitions $(V_1,\ldots,V_{n-k})$ or
equivalently by $(n-k)$-weak orders on $V$. Accordingly, to any
face $F\subset Pe^{n-1}$ is associated the monomial quasisymmetric
function $M_F$, where

$$M_F=M_{{\rm type}(\precsim_F)},$$ for the weak order $\precsim_F$ on
$V$corresponding to the face $F$. Specially, facets correspond to
pairs $(A,V\setminus A)$, for proper subsets $A\subset V$ and the
associated monomial quasissymetric functions are of the form
$M_{(k,n-k)}$ for $1\leq k\leq n$. Vertices correspond to linear
orders $v_{i_1}<\ldots<v_{i_n}$ on $V$ with associated monomial
quasisymmetric functions equal to $M_{(1,\ldots,1)}$.



By Proposition \ref{cones} the normal cone at the vertex $v\in
Pe^{n-1}$ that corresponds to a permutation
$\pi_v=(i_1,\ldots,i_n)$ is the Weyl chamber

$$\sigma_v=C_{\pi_v}: x_{i_1}<\cdots<x_{i_n}.$$
The braid arrangement $\mathcal{A}_n$ is the arrangement of
hyperplanes $$\mathcal{A}_n: x_i=x_j, 1\leq i,j\leq n$$ in the
quotient space
$\mathbb{R}^{n}/\mathbb{R}\cdot(1,\ldots,1)\cong\mathbb{R}^{n-1}$.
The normal fan $\Sigma_{Pe^{n-1}}$ of the permutohedron is the
simplicial fan defined by $\mathcal{A}_n$. A braid cone is the
polyhedral cone given by the conjuction of inequalities of the
form $x_i\leq x_j.$ There is an obvious bijection between
preorders $\precsim$ on $[n]$ and braid cones determined by
equivalency $x_i\leq x_j$ if and only if $i\precsim j$. The
correspondence and properties of preorders and braid cones are
given in \cite[Proposition 3.5]{PRW}. We remark that linear orders
on $[n]$ correspond to full-dimensional braid cones. The monomial
quasisymmetric function $M_F$ is precisely the enumerator for all
positive lattice points in the interior of the normal cone
associated to the face $F\subset Pe^{n-1}$.

For each generalized permutohedron $Q$ there is a map
$\Psi_Q:S_n\rightarrow{\rm Vertices}(Q)$ defined by $\Psi(\pi)=v$
if and only if the normal cone $\sigma_v$ of $Q$ at $v$ contains
the Weyl chamber $C_\pi$ or equivalently the permutation $\pi\in
S_n$ is a linear extension of the poset determined by the normal
cone at $v$ \cite[Corollary 3.9]{PRW}.

\section{Hopf algebra morphism}

The goal of this section is to show that the assignment of
quasisymmetric function $F(P_B)$ to a building set $B$ is a Hopf
algebra morphism. We use the theory of combinatorial Hopf algebras
developed in the originating paper by Aguiar, Bergeron and Sottile
\cite{ABS}. For an extensive survey of the theory see also
\cite{GR}. A combinatorial Hopf algebra $(\mathcal{H},\zeta)$ over
a field $\mathbf{k}$ is a graded, connected Hopf algebra equipped
with a multiplicative linear functional
$\zeta:\mathcal{H}\rightarrow\mathbf{k}$ called a character.

We construct a graded Hopf algebra associated with the species of
building sets in the sense of \cite{Sch}. Let $\mathcal{B}$ be the
graded vector space generated by the set of all isomorphism
classes of building sets. The grading is defined by the number of
vertices.

For a building set $B$ on $[n]$ and a subset $I\subset [n]$, let
$B\mid_I=\{J\subset I\mid J\in B\}$ be the induced building
subset. The contraction of $I\subset[n]$ from $B$ is the building
set on $[n]\setminus I$ defined by $B/I=\{J\subset [n]\setminus
I\mid J\in B \ {\rm or} \ I'\cup J\in B\ \ \mbox{for some} \
I'\subset I\}$. Define the multiplication and comultiplication by

\[B_1\cdot B_2=B_1\sqcup B_2 \ \ \mbox{and} \ \
\Delta(B)=\sum_{I\subset V}B\mid_I\otimes B/I.\] The unit is the
building set $B_\emptyset$ on the empty set and the counit is
defined by $\epsilon(B_\emptyset)=1$ and zero otherwise.

\begin{proposition}
The vector space $\mathcal{B}$ with the above defined operations
is a graded commutative and non-cocommutative connected bialgebra.
\end{proposition}

\begin{proof}
The only nontrivial parts of the statement are the coassociativity
and the compatibility of operations, which follows from the
straightforward identities $(B/I)\mid_J=(B\mid_{I\sqcup J})/I,
(B/I)/J=B/(I\sqcup J)$ for any disjoint $I,J\subset V$  and
$(B_1\cdot B_2)\mid_{I_1\sqcup I_2}=B_1\mid_{I_1}\cdot
B_2\mid_{I_2}, (B_1\cdot B_2)/(I_1\sqcup I_2)=B_1/I_1\cdot
B_2/I_2$ for all $I_1\subset V_1, I_2\subset V_2$.
\end{proof}

The antipode of $\mathcal{B}$ is determined by general Takeuchi's
formula for the antipode of a graded connected bialgebra
(\cite[Lemma 14]{T}, see also \cite[Proposition 1.44]{GR})

\[S(B)=\sum_{k\geq
1}(-1)^k\sum_{\mathcal{L}_k}\prod_{j=1}^{k}(B\mid_{I_j})/I_{j-1},\]
where the inner sum goes over all chains of subsets
$\mathcal{L}_k:\emptyset=I_0\subset I_1\subset\cdots\subset
I_{k-1}\subset I_k=V$.

Define a character $\zeta:\mathcal{B}\rightarrow\mathbf{k}$ by
$\zeta(B)=1$ if $B$ is discrete and zero otherwise. This
determines the combinatorial Hopf algebra $(\mathcal{B},\zeta)$.

\begin{remark}
Another combinatorial Hopf algebra of building set $BSet$, which
is a Hopf subalgebra of the chromatic Hopf algebra of hypergraphs
is studied in \cite{GS}, \cite{GSJ}. As algebras $\mathcal{B}$ and
$BSet$ are the same but the coalgebra structures are different.
This is reflected in the fact that $BSet$ is cocommutative, in
opposite to $\mathcal{B}$.
\end{remark}
\begin{remark}
The algebra $\mathcal{B}$ has an additional structure of a
differential algebra introduced in \cite{B}. The derivation is
determined by
\[d(B)=\sum_{I\in B\setminus\{[n]\}}B\mid_I\cdot B/I\] for connected building set on
$[n]$ and extended by Leibniz law $d(B_1B_2)=d(B_1)B_2+B_1d(B_2)$.
\end{remark}

\begin{definition}\label{splittingchains}
Given a composition $\alpha=(a_1,\ldots,a_k)\models n$, we say
that the chain $\mathcal{L}:\emptyset=I_0\subset
I_1\subset\cdots\subset I_{k-1}\subset I_k=V$ is a {\it splitting
chain} of the type ${\rm type}(\mathcal{L})=\alpha$ of a building
set $B$ if $(B\mid_{I_j})/I_{j-1}$ is discrete and $|I_j\setminus
I_{j-1}|=a_j$ for all $1\leq j\leq k$. A splitting chain
$\mathcal{L}$ determines the weak order
$\preceq_\mathcal{L}=(I_1,I_2\setminus I_1,\ldots,I_k\setminus
I_{k-1})$ on $V$ of the same type.
\end{definition}

\begin{theorem}\label{universal}
For a connected building set $B$ the generating function $F(P_B)$
has the following expansion
\[F(P_B)=\sum_{\alpha\models n}\zeta_\alpha(B)M_\alpha,\]
where $\zeta_\alpha(B)$ is the total number of splitting chains of
the type $\alpha$.
\end{theorem}

\begin{proof}
We define a map $g:\Lambda\rightarrow \mathrm{Vertices}(P_B)$ from
the set $\Lambda$ of splitting chains of $B$ to the set of
vertices of $P_B$.

Let $\mathcal{L}:\emptyset=I_0\subset I_1\subset\cdots\subset
I_{k-1}\subset I_k=V=[n]$ be a splitting chain of $B$. Define the
level of a vertex $i\in V$ by $l(i)=j$ if $i\in I_j\setminus
I_{j-1}$ and a map $S:V\rightarrow B$ by $S(i)=\max\{J\in
B\mid_{I_{l(i)}}\mid i\in J\}, i\in V$. The map $S$ is well
defined and $S(i)\setminus\{i\}\subset I_{l(i)-1}$ since
$(B\mid_{I_{l(i)}})/I_{l(i)-1}$ is discrete. Particulary
$S(i)=\{i\}$ for each $i\in V$ such that $l(i)=1$ and $S(i)=V$ for
the unique $i\in V$. Let $N(\mathcal{L})=\{S(i)\mid i\in
V\}\setminus\{V\}$. We check the conditions of Proposition
\ref{nested} to show that the collection $N(\mathcal{L})$ is a
maximal nested set.

\begin{itemize}

\item[(N1)] Suppose that $S(i)\cap S(j)\neq\emptyset$ for some
$i,j\in V$. It implies that $S(i)\cup S(j)\in B$. If $l=l(i)=l(j)$
then $\{i,j\}\in (B\mid_{I_{l}})/I_{l-1}$ which contradicts the
condition that $(B\mid_{I_{l}})/I_{l-1}$ is discrete. If
$l(j)<l(i)$ then $i\in S(i)\cup S(j)\in B$ which implies
$S(j)\subset S(i)$ by definition of $S(i)$.

\item[(N2)] For a collection $S(i_1),\ldots,S(i_p)$ such that
$S=S(i_1)\cup\ldots\cup S(i_p)\in B$ we have by definition that
$S=S(i_j)$ for a vertex $i_j\in V$ with the maximal level
$l=\max\{l(i_1),\ldots,l(i_p)\}$. Thus $S(i_1),\ldots,S(i_p)$ can
not be a pairwise disjoint collection if $p>1$.
\end{itemize}

The map $g$ is given by $g(\mathcal{L})=v$ if
$N(\mathcal{L})=N_v$. It is well defined since vertices of $P_B$
and maximal nested sets are in one-to-one correspondence. We show
the following identity

$$\sum_{f\in\sigma_v}\mathbf{x}_f=\sum_{\mathcal{L}:g(\mathcal{L})=v}M_{{\rm
type}(\mathcal{L})}.$$

Let $\mathcal{L}$ be a splitting chain such that
$g(\mathcal{L})=v$. Then the associated level function satisfies
$l(i)<l(j)$, for each $S(i)\subset S(j)\ \mbox{in} \
N(\mathcal{L})$. By Proposition \ref{cones} (ii) we have
$l\in\sigma_v$ which shows that the monomial quasisymmetric
function $M_{{\rm type}(\mathcal{L})}$ is a summand of
$\sum_{f\in\sigma_v}\mathbf{x}_f$.

On the other hand, for $f\in\sigma_v$ with the set of values
$i_1<\cdots<i_k$, let $\mathcal{L}_f:\emptyset\subset
I_1\subset\cdots\subset I_k$ be a chain, where $I_j=\{i\in V\mid
f(i)\leq i_j\}, 1\leq j\leq k$. We can convince that the chain
$\mathcal{L}_f$ is a splitting chain of $B$ such that
$N(\mathcal{L}_f)=N_v$, which implies that the monomial
$\mathbf{x}_f$ appears in $M_{\mathrm{type}(\mathcal{L}_f)}$.


Finally, as

\[F(P_B)=\sum_{v\in P_B}\sum_{f\in\sigma_v}\mathbf{x}_f=\sum_{v\in
P_B}\sum_{\mathcal{L}\in g^{-1}(v)}M_{{\rm
type}(\mathcal{L})}=\sum_{\alpha\models
n}\zeta_\alpha(B)M_\alpha,\] the theorem is proved.
\end{proof}

\begin{remark}
Definition \ref{splittingchains} of splitting chains is borrowed
from \cite{BJR} where it appears in the context of matroids.
Theorem \ref{universal} is analogous to \cite[Proposition
3.3]{BJR}.
\end{remark}

The character $\zeta_Q:QSym\rightarrow\mathbf{k}$, defined on the
monomial basis by $\zeta_Q(M_\alpha)=1$ for either $\alpha=()$ or
$\alpha=(n)$ and zero otherwise, turns the Hopf algebra of
quasisymmetric functions $QSym$ into the terminal object in the
category of combinatorial Hopf algebras over a field $\mathbf{k}$
\cite[Theorem 4.1]{ABS}. This means that for each combinatorial
Hopf algebra $(\mathcal{H},\zeta)$ there is a unique morphism of
combinatorial Hopf algebras
$\Psi:(\mathcal{H},\zeta)\rightarrow(QSym,\zeta_Q)$. The explicit
definition of this morphism implies the following corollary of
Theorem \ref{universal}.

\begin{corollary}\label{main}
The map $F:\mathcal{B}\rightarrow QSym$, defined by $F(B)=F(P_B)$,
is a morphism of combinatorial Hopf algebras.
\end{corollary}

\begin{proof}
Let $p_j:\mathcal{B}\rightarrow\mathcal{B}_j$ be the projection on
the homogeneous part of degree $j$. The morphism
$\Psi:(\mathcal{B},\zeta)\rightarrow(Qsym,\zeta_Q)$ is defined by
\[\Psi(B)=\sum_{\alpha\models n}p_\alpha(B)M_\alpha,\] where
$p_\alpha=p_{(a_1,\ldots,a_k)}=p_{a_1}\ast\ldots\ast
p_{a_k}=m^{(k-1)}\circ(p_{a_1}\otimes\ldots\otimes
p_{a_k})\circ\Delta^{(k-1)}$ is the convolution product of
projections. It is straightforward to convince that
$p_\alpha(B)=\zeta_\alpha(B)$ for any composition $\alpha\models
n$, so the morphism $\Psi$ coincides with the map $F$.
\end{proof}

As a consequence we obtain the following identities for the
function $F$:

\[F(P_{B_1}\times P_{B_2})=F(P_{B_1})F(P_{B_2}),\]
\[\Delta(F(P_B))=\sum_{I\subset V}F(P_{B\mid_I})\otimes
F(P_{B/I}).\]

\begin{remark}
The function $F(P_B)$ is not a combinatorial invariant of
nestohedra. For example, the building sets
$B_1=\{1,2,3,4,12,123\}$ and $B_2=\{1,2,3,4,12,34\}$ on the four
element set $V=[4]$ have $P_{B_1}$ and $P_{B_2}$ combinatorially
equivalent to the $3$-cube, but $F(B_1)\neq F(B_2)$.
\end{remark}

\section{Unlabeled rooted trees}

Let $T$ be an unlabeled rooted tree on the set of vertices
$V=\{v_1,\ldots,v_n\}$. It defines a poset $(V,\leq_T)$ with
$v_i\leq v_j$ if and only if $v_j$ is the node on the unique path
from $v_i$ to the root. We do not make a difference between the
rooted tree $T$ and the corresponding Hasse diagram of the poset
$(V,\leq_T)$.

\begin{remark}
Let $\mathcal{T}_n$ be the set of all unlabeled rooted trees on
$n$ nodes and $r(n)$ be the total number of elements of
$\mathcal{T}_n$. In Neil Sloan's OEIS the sequence
$\{r(n)\}_{n\in\mathbb{N}}$ is numerated by A000081.
\end{remark}

We need some basic notions from Stanley's theory of
$P$-partitions. A detailed survey of the theory can be found in
\cite{S}, \cite{GR}. Let $f:T\rightarrow\mathbb{N}$ be a function
on vertices of a rooted tree $T$. We call it {\it natural
$T$-partition} if $f(v_i)\leq f(v_j)$ for $v_i\leq v_j$ in $T$ and
{\it strict $T$-partition} if in addition $f(v_i)<f(v_j)$ for any
pair of vertices with $v_i<v_j$ in $T$. Write $\mathcal{A}(T)$ for
the set of all natural $T$-partitions and $\mathcal{A}_0(T)$ for
its subset of strict $T$-partitions. Let $F(T)$ be the
quasisymmetric enumerator of strict $T$-partitions

$$F(T)=\sum_{f\in\mathcal{A}_0(T)}\mathbf{x}_f.$$


\begin{example}\label{fourtree}
There are four unlabeled rooted trees on four vertices whose
corresponding enumerators $F(T)$ in the monomial basis are given
by

\begin{tikzpicture}[scale=.5]
\node [left] at (0,0) {$F\Big($};\draw
(.5,-1.5)--(.5,-.5)--(.5,.5)--(.5,1.5);\draw[fill](.5,-1.5)circle(.1cm);
\draw[fill](.5,-.5)circle(.1cm);\draw[fill](.5,.5)circle(.1cm);
\draw[fill](.5,1.5)circle(.1cm); \node[right]at
(1,0){\small$\Big)=M_{(1,1,1,1)}$,};\node[left] at
(6.5,0){$F\Big($};

\draw(6.5,-.5)--(7.5,.5)--(8.5,-.5);\draw(7.5,.5)--(7.5,-.5);\draw[fill](6.5,-.5)circle(.1cm);
\draw[fill](7.5,.5)circle(.1cm);\draw[fill](8.5,-.5)circle(.1cm);\draw[fill](7.5,-.5)circle(.1cm);
\node[right]at(8.5,0){\small$\Big)=6M_{(1,1,1,1)}+3M_{(2,1,1)}+3M_{(1,2,1)}+M_{(3,1)},$};

\end{tikzpicture}

\begin{tikzpicture}[scale=.5]

\node[left]at (0,0){$F\Big($};
\draw(.5,-1)--(.5,0)--(1,1)--(1.5,0);\draw[fill](.5,-1)circle(.1cm);
\draw[fill](.5,0)circle(.1cm);\draw[fill](1,1)circle(.1cm);
\draw[fill](1.5,0)circle(.1cm);\node[right]at(2,0){\small$\Big)=3M_{(1,1,1,1)}+M_{(2,1,1)}+M_{(1,2,1)}$,};
\node[left]at(14,0){$F\Big($};
\draw(14,-1)--(14.5,0)--(14.5,1);\draw(14.5,0)--(15,-1);\draw[fill](14,-1)circle(.1cm);
\draw[fill](14.5,0)circle(.1cm);\draw[fill](14.5,1)circle(.1cm);\draw[fill](15,-1)circle(.1cm);
\node[right]at(15,0){\small$\Big)=2M_{(1,1,1,1)}+M_{(2,1,1)}$.};

\end{tikzpicture}

\end{example}

The quasisymmetric function $F(T)$ can be determined recursively.
To each vertex $v\in V$ define $T_{\leq v}$ as the complete
subtree on the set $\{u\in V\mid u\leq v\}$ of predecessors of
$v$. The leaf is a vertex $v\in V$ for which $T_{\leq v}=\{v\}$.
For a rooted forest $T=\sqcup_{i=1}^{k} T_i$ which is a finite
collection of rooted trees we extend multiplicatively

$$F(\sqcup_{i=1}^{k} T_i)=F(T_1)\cdots F(T_k).$$

\begin{definition}\label{shift}
A shifting operator $F\mapsto(F)_1$ on quasisymmetric functions is
the linear extension of the map defined on the monomial basis by
$(M_\alpha)_1=M_{(\alpha,1)},$ for each composition $\alpha$.
\end{definition}

\begin{theorem}\label{tree}
Given an unlabeled rooted tree $T$ on the set of vertices $V$ with
the root $v_0\in V$. Let $T_1,\ldots,T_k$ be connected components
of the forest $T\setminus\{v_0\}$. Then

$$F(T)=(\prod_{i=1}^{k}F(T_i))_1=F(T\setminus\{v_0\})_1.$$

\end{theorem}

\begin{proof}
Denote by $v_1,\ldots,v_k$ the neighbors in $T$ of the root $v_0$.
Then $T_i=T_{\leq v_i}$ for $i=1,\ldots,k$. A function
$f:T\rightarrow\mathbb{N}$ is a $T$-partition if and only if its
restrictions $f\mid_{T_i}:T_i\rightarrow\mathbb{N}$ are
$T_i$-partitions for all $i=1,\ldots,k$ and $f(v)<f(v_0)$ for each
$v\neq v_0$.
\end{proof}

Each $T$-partition $f:T\rightarrow\mathbb{N}$ takes the maximal
value at the root of $T$. Therefore each monomial function
$M_\alpha$ in the expansion of $F(T)$ in the monomial basis is
indexed by the composition $\alpha$ whose last component is $1$.
Since $r(n)>2^{n-2}={\rm dim}(QSym_{n-1})$ for $n>4$, we proved
the following

\begin{proposition}
The quasisymmetric functions $\{F(T)\}_{T\in\mathcal{T}_n}$ are
linearly dependent for each $n>4$.
\end{proposition}

\begin{example}\label{dependence}
We have $r(5)=9$ and ${\rm dim}(QSym_4)_1=8$. The unique linear
dependence relation is given by

\begin{tikzpicture}[scale=.5]
\node [left] at (-1,0) {$F\Big($}; \draw (-1,-1)--(-.5,0)--(0,1);
\draw (-.5,0)--(0,-1); \draw (0,1)--(.5,0); \draw [fill] (-1,-1)
circle (.1cm); \draw [fill] (-.5,0) circle (.1cm); \draw [fill]
(0,1) circle (.1cm);\draw [fill] (.5,0) circle (.1cm);\draw [fill]
(0,-1) circle (.1cm);\node [right] at (1,0){$\Big)+2F\Big($};\draw
(4.5,-1.5)--(4.5,-.5)--(5,.5)--(5,1.5);\draw
(5,.5)--(5.5,-.5);\draw [fill] (4.5,-1.5) circle (.1cm);\draw
[fill] (4.5,-.5) circle (.1cm);\draw [fill] (5,.5) circle
(.1cm);\draw [fill] (5,1.5) circle (.1cm);\draw [fill] (5.5,-.5)
circle (.1cm);\node [right] at (6,0){$\Big)=F\Big($};\draw
(9,-1)--(10,0)--(10,1);\draw (10,0)--(10,-1);\draw
(10,0)--(11,-1);\draw [fill] (9,-1) circle (.1cm);\draw [fill]
(10,-1) circle (.1cm);\draw [fill] (11,-1) circle (.1cm);\draw
[fill] (10,0) circle (.1cm);\draw [fill] (10,1) circle
(.1cm);\node [right] at (11,0){$\Big)+F\Big($};\draw
(14,-1)--(14,0)--(14.5,1)--(15,0)--(15,-1);\draw [fill] (14,-1)
circle(.1cm);\draw [fill] (14,0) circle(.1cm);\draw [fill]
(14.5,1) circle(.1cm);\draw [fill] (15,0) circle(.1cm);\draw
[fill] (15,-1) circle(.1cm);\node[right] at
(15,0){\Big)+F\Big(};\draw
(17.5,-1.5)--(18,-.5)--(18,.5)--(18,1.5);\draw
(18,-.5)--(18.5,-1.5);\draw [fill] (17.5,-1.5) circle(.1cm);\draw
[fill] (18,-.5) circle(.1cm);\draw [fill] (18,.5)
circle(.1cm);\draw [fill] (18,1.5) circle(.1cm);\draw [fill]
(18.5,-1.5) circle(.1cm);\node[right] at (18.5,0) {\Big).};
\end{tikzpicture}
\end{example}


Given a connected building set $B$, recall that to each vertex
$v\in P_B$ is associated the labeled rooted tree $T_v$, called
$B$-tree, which is the Hasse diagram of the poset $P_v$. Denote by
$T^{\circ}_v$ the unlabeled rooted tree associated to a $B$-tree
$T_v$ by forgetting the labels. Let $T(B)=\{T^{\circ}_v\mid v\in
P_B\}$ be the multiset of the corresponding unlabeled rooted
trees. The following expansion is a special case of \cite[Theorem
9.2]{BJR} which holds for generalized permutohedra.

\begin{theorem}\label{ppart}
For a building set $B$ the quasisymmetric enumerator $F(P_B)$ is
the sum of $T$-partitions enumerators corresponding to vertices of
$P_B$

$$F(P_B)=\sum_{T\in T(B)}F(T).$$
\end{theorem}
\begin{proof}
It is sufficient to show the identity
$F(T^{\circ}_v)=\sum_{f\in\sigma_v}\mathbf{x}_f$ which follows
from the description of the normal cone $\sigma_v$ at a vertex
$v\in P_B$, see Proposition \ref{cones} (ii).
\end{proof}

\begin{corollary}
The quasisymmetric function $F(P_B)$ depends only on the multiset
$T(B)$ of unlabeled rooted trees $T^{\circ}_v$ corresponding to
the vertices $v\in P_B$.
\end{corollary}

\begin{question}
In what extent the multiset $T(B)$ determines a building set $B$?
\end{question}

We say that a weak order $\preceq$ extends a partial order
$P=([n],\leq)$ and write $P\subset\preceq$ if $i<j$ implies
$i\prec j$ for each $i,j\in[n]$. Theorem \ref{ppart} implies the
following interpretation

$$F(P_B)=\sum_{v\in
P_v}\sum_{P_v\subset\preceq}M_{\mathrm{type}(\preceq)}.$$ Since
weak orders on the set $[n]$ and faces of the permutohedron
$Pe^{n-1}$ are in one-to-one correspondence, we associate to each
vertex $v\in P_B$ the collection of faces $F\subset Pe^{n-1}$ such
that the weak order $\preceq_F$ corresponding to a face $F$
extends the partial order $P_v$ corresponding to the vertex $v$.
This is exactly the collection of faces that collapses to the
vertex $v\in P_B$ by deforming the permutohedron $Pe^{n-1}$ to the
nestohedron $P_B$.


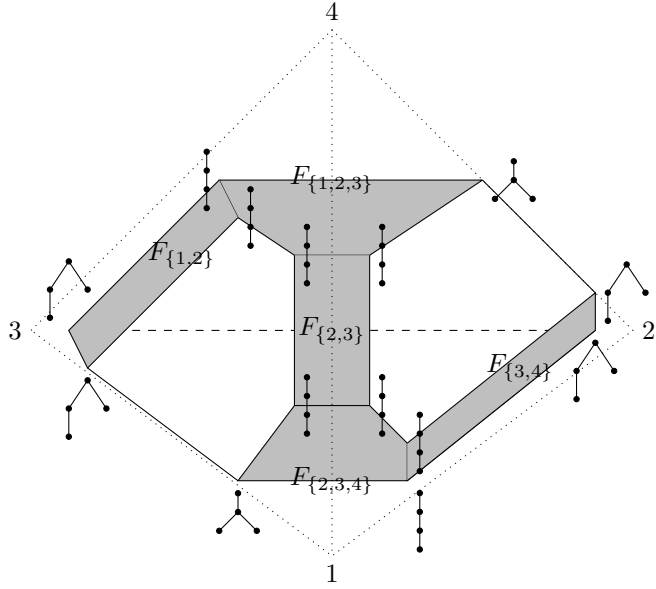
\begin{figure}[h!h!]\centering
\begin{tikzpicture}[scale=.5]
\draw[dotted](-8,0)--(0,8)--(8,0)--(0,-6)--(-8,0);
\draw[dashed](-7,0)--(7,0);
\node[left]at(-8,0){3};\node[right]at(8,0){2};
\node[below]at(0,-6){1};\node[above]at(0,8){4};
\draw[fill=lightgray](4,4)--(-3,4)--(-2.5,3)--(-1,2)--(1,2)--(4,4);
\draw[fill=lightgray](-3,4)--(-7,0)--(-6.5,-1)--(-2.5,3);\draw[fill=lightgray](-1,2)--(-1,-2)--(-2.5,-4)--(2,-4)--(2,-3)--(1,-2)--(1,2);
\draw(-1,-2)--(1,-2);\draw(2,-4)--(7,0);\draw[fill=lightgray](2,-3)--(7,1)--(7,0)--(2,-4);\draw(-6.5,-1)--(-2.5,-4);\draw(4,4)--(7,1);
\draw[dotted](0,8)--(0,-6);

\node[right]at(4,4){\begin{tikzpicture}[scale=.5]\draw(-.5,-.5)--(0,0)--(0,.5);\draw(0,0)--(.5,-.5);\draw[fill](-.5,-.5)circle(2pt);
\draw[fill](0,0)circle(2pt);\draw[fill](0,.5)circle(2pt);\draw[fill](.5,-.5)circle(2pt);\end{tikzpicture}};

\node[right]at(-1,2){\begin{tikzpicture}[scale=.5]\draw(0,-.75)--(0,-.25)--(0,.25)--(0,.75);\draw[fill](0,-.75)circle(2pt);
\draw[fill](0,-.25)circle(2pt);\draw[fill](0,.25)circle(2pt);\draw[fill](0,.75)circle(2pt);\end{tikzpicture}};

\node[right]at(1,2){\begin{tikzpicture}[scale=.5]\draw(0,-.75)--(0,-.25)--(0,.25)--(0,.75);\draw[fill](0,-.75)circle(2pt);
\draw[fill](0,-.25)circle(2pt);\draw[fill](0,.25)circle(2pt);\draw[fill](0,.75)circle(2pt);\end{tikzpicture}};

\node[right]at(1,-2){\begin{tikzpicture}[scale=.5]\draw(0,-.75)--(0,-.25)--(0,.25)--(0,.75);\draw[fill](0,-.75)circle(2pt);
\draw[fill](0,-.25)circle(2pt);\draw[fill](0,.25)circle(2pt);\draw[fill](0,.75)circle(2pt);\end{tikzpicture}};

\node[right]at(-1,-2){\begin{tikzpicture}[scale=.5]\draw(0,-.75)--(0,-.25)--(0,.25)--(0,.75);\draw[fill](0,-.75)circle(2pt);
\draw[fill](0,-.25)circle(2pt);\draw[fill](0,.25)circle(2pt);\draw[fill](0,.75)circle(2pt);\end{tikzpicture}};

\node[right]at(-2.5,3){\begin{tikzpicture}[scale=.5]\draw(0,-.75)--(0,-.25)--(0,.25)--(0,.75);\draw[fill](0,-.75)circle(2pt);
\draw[fill](0,-.25)circle(2pt);\draw[fill](0,.25)circle(2pt);\draw[fill](0,.75)circle(2pt);\end{tikzpicture}};

\node[right]
at(2,-3){\begin{tikzpicture}[scale=.5]\draw(0,-.75)--(0,-.25)--(0,.25)--(0,.75);\draw[fill](0,-.75)circle(2pt);
\draw[fill](0,-.25)circle(2pt);\draw[fill](0,.25)circle(2pt);\draw[fill](0,.75)circle(2pt);\end{tikzpicture}};

\node[below
right]at(2,-4){\begin{tikzpicture}[scale=.5]\draw(0,-.75)--(0,-.25)--(0,.25)--(0,.75);\draw[fill](0,-.75)circle(2pt);
\draw[fill](0,-.25)circle(2pt);\draw[fill](0,.25)circle(2pt);\draw[fill](0,.75)circle(2pt);\end{tikzpicture}};

\node[left]at(-3,4){\begin{tikzpicture}[scale=.5]\draw(0,-.75)--(0,-.25)--(0,.25)--(0,.75);\draw[fill](0,-.75)circle(2pt);
\draw[fill](0,-.25)circle(2pt);\draw[fill](0,.25)circle(2pt);\draw[fill](0,.75)circle(2pt);\end{tikzpicture}};

\node[below]at(-2.5,-4){\begin{tikzpicture}[scale=.5]\draw(-.5,-.5)--(0,0)--(0,.5);\draw(0,0)--(.5,-.5);\draw[fill](-.5,-.5)circle(2pt);
\draw[fill](0,0)circle(2pt);\draw[fill](0,.5)circle(2pt);\draw[fill](.5,-.5)circle(2pt);\end{tikzpicture}};

\node[below]at(-6.5,-1){\begin{tikzpicture}[scale=.5]\draw(.5,-.75)--(.5,0)--(1,.75)--(1.5,0);\draw[fill](.5,-.75)circle(2pt);
\draw[fill](.5,0)circle(2pt);\draw[fill](1,.75)circle(2pt);
\draw[fill](1.5,0)circle(2pt);\end{tikzpicture}};

\node[below]at(7,0){\begin{tikzpicture}[scale=.5]\draw(.5,-.75)--(.5,0)--(1,.75)--(1.5,0);\draw[fill](.5,-.75)circle(2pt);
\draw[fill](.5,0)circle(2pt);\draw[fill](1,.75)circle(2pt);
\draw[fill](1.5,0)circle(2pt);\end{tikzpicture}};

\node[right]at(7,1){\begin{tikzpicture}[scale=.5]\draw(.5,-.75)--(.5,0)--(1,.75)--(1.5,0);\draw[fill](.5,-.75)circle(2pt);
\draw[fill](.5,0)circle(2pt);\draw[fill](1,.75)circle(2pt);
\draw[fill](1.5,0)circle(2pt);\end{tikzpicture}};

\node[above]at(-7,0){\begin{tikzpicture}[scale=.5]\draw(.5,-.75)--(.5,0)--(1,.75)--(1.5,0);\draw[fill](.5,-.75)circle(2pt);
\draw[fill](.5,0)circle(2pt);\draw[fill](1,.75)circle(2pt);
\draw[fill](1.5,0)circle(2pt);\end{tikzpicture}};

\node at (0,4){$F_{\{1,2,3\}}$}; \node at(0,0){$F_{\{2,3\}}$};
\node at(0,-4){$F_{\{2,3,4\}}$}; \node at(-4,2){$F_{\{1,2\}}$};
\node at(5,-1){$F_{\{3,4\}}$};

\end{tikzpicture}
\caption{Associahedron $As^{3}$} \label{associahedron}
\end{figure}

\begin{example}\label{assoc}
The $3$-dimensional associahedron $As^{3}$ is realized as the
graph-asso\-ciahedron $P_{B(L_4)}$ corresponding to the path graph
$L_4$ on the set of vertices $[4]$. The determining building set
is $B(L_4)=\{1,2,3,4,12,23,34,123,234,1234\}$. To illustrate the
general construction we describe in more details how the unlabeled
rooted trees $T^{\circ}_v$ are associated to the vertices $v\in
As^{3}$, see Figure \ref{associahedron}. The construction depends
only on the combinatorial type of a nestohedron $P_B$. Therefore
we can start with a $3$-simplex $\Delta^{3}$ with the faces
$\Delta_I$ labeled by subsets $I\subset[4]$. Performing the
truncations along $\Delta_{[4]\setminus J}$ for $J\in
B(L_4)\setminus\{[4]\}$ in nondecreasing order of dimensions,
produces $As^{3}$ with the facets $F_J$. Each vertex is an
intersection of the form $v=F_{J_1}\cap F_{J_2}\cap F_{J_3}$,
where the collection $N_v=\{J_1,J_2,J_3\}$ is a maximal nested
set. The rooted tree $T_v$ is the Hasse diagram of the poset $P_v$
and $T^{\circ}_v$ is the corresponding unlabeled rooted tree.

By Theorem \ref{ppart} and Example \ref{fourtree} we find

\[F(As^{3})=4M_{(1,2,1)}+6M_{(2,1,1)}+24M_{(1,1,1,1)}.\]

\end{example}

\section{The action of the antipode on $F(P_B)$}

The action $F^{\ast}(Q)=S(F(Q))$ of the antipode $S$ of
quasisymmetric functions on the lattice points enumerator $F(Q)$
is determined for a general class of generalized permutohedra in
\cite[Theorem 9.2]{BJR}. We consider this formula for a special
class of nestohedra.

The formula for the antipode in the monomial basis are obtained
independently in \cite[Corollary 2.3]{MR}, \cite[Proposition
3.4]{E}, see also \cite[Theorem 5.11]{GR}. The antipode formula on
$P$-partition enumerators (see \cite[Theorem 3.1]{MR1} and
\cite[Corollary 5.27]{GR}) has a particularly nice interpretation
for unlabeled rooted trees. Let $T^{\mathrm{op}}$ be the unlabeled
rooted tree $T$ with the reverse order of vertices. Thus the root
is the minimal vertex in $T^{\mathrm{op}}$. Denote by
$\widehat{F}(T^{\mathrm{op}})$ the quasisymmetric enumerator of
natural $T^{\mathrm{op}}$-partitions

$$\widehat{F}(T^{\mathrm{op}})=\sum_{f\in\mathcal{A}(T^{\mathrm{op}})}\mathbf{x}_f.$$
Then the following formula holds

$$S(F(T))=(-1)^{n}\widehat{F}(T^{\mathrm{op}}).$$ For example $S\Big(F\Big($\begin{tikzpicture}[baseline={(current bounding box)},scale=.5]
\draw (-.5,-.5)--(0,.5)--(.5,-.5); \draw [fill] (-.5,-.5) circle
(.1cm); \draw [fill] (0,.5) circle (.1cm); \draw [fill] (.5,-.5)
circle
(.1cm);\end{tikzpicture}$\Big)\Big)=-\widehat{F}\Big($\begin{tikzpicture}[baseline={(current
bounding box)},scale=.5] \draw (4,.5)--(4.5,-.5)--(5,.5);\draw
[fill] (4,.5) circle (.1cm); \draw [fill] (4.5,-.5) circle (.1cm);
\draw [fill] (5,.5) circle (.1cm);
\end{tikzpicture}$\Big)$, where $F\Big($\begin{tikzpicture}[baseline={(current bounding box)},scale=.5]
\draw (-.5,-.5)--(0,.5)--(.5,-.5); \draw [fill] (-.5,-.5) circle
(.1cm); \draw [fill] (0,.5) circle (.1cm); \draw [fill] (.5,-.5)
circle (.1cm);\end{tikzpicture}$\Big)=2M_{(1,1,1)}+M_{(2,1)}$ and
$\widehat{F}\Big($\begin{tikzpicture}[baseline={(current bounding
box)},scale=.5] \draw (4,.5)--(4.5,-.5)--(5,.5);\draw [fill]
(4,.5) circle (.1cm); \draw [fill] (4.5,-.5) circle (.1cm); \draw
[fill] (5,.5) circle (.1cm);
\end{tikzpicture}$\Big)=2M_{(1,1,1)}+2M_{(2,1)}+M_{(1,2)}+M_{(3)}.$
Consequently Theorem \ref{ppart} implies

$$F^{\ast}(P_B)=(-1)^{n}\sum_{T\in
T(B)}\widehat{F}(T^{\mathrm{op}}).$$

\begin{example}
Let $As^{3}=P_{B(L_4)}$ as in Example \ref{assoc}. We find
$F^{\ast}(As^{3})=14M_{(4)}+14M_{(1,3)}+20M_{(3,1)}+18M_{(2,2)}+18M_{(1,1,2)}+20M_{(1,2,1)}+24M_{(2,1,1)}+24M_{(1,1,1,1)}$.
\end{example}

The following proposition, proved in the case of matroid base
polytopes \cite[Theorem 6.3]{BJR} and stated in \cite[Theorem
9.2]{BJR} for generalized permutohedra, gives a combinatorial
interpretation of the action of the antipode $S$ on $F(P_B)$.

\begin{proposition}\label{fstar}
The quasisymmetric function $F^{\ast}(P_B)=S(F(P_B))$ is the
enumerator function
$$F^{\ast}(P_B)=(-1)^{n}\sum_{f}c(f)\mathbf{x}_f,$$ where the sum is over
all $f:[n]\rightarrow\mathbb{N}$ and $c(f)=|\{v\in P_B\mid
\displaystyle{\min_{x\in P_B}}\langle f,x\rangle=\langle
f,v\rangle\}|$ is the total number of vertices $v\in P_B$ that
minimize a function $f$.
\end{proposition}
\begin{proof}

Let $\sigma_v^{\mathrm{op}}$ be the opposite cone to the normal
cone $\sigma_v$ at a vertex $v\in P_B$. By Proposition \ref{cones}
(ii) we have that the closure of the opposite cone
$\overline{\sigma_v^{\mathrm{op}}}$ is determined by inequalities

$$x_j\leq x_i, \
\mbox{for all} \ v_i<v_j \in T_v.$$ Therefore a function
$f:[n]\rightarrow\mathbb{N}$ is a natural
$T^{\mathrm{op}}$-partition $f\in\mathcal{A}(T^{\mathrm{op}})$ if
and only if it belongs to the closure of the opposite cone
$f\in\overline{\sigma_v^{\mathrm{op}}}$ at some vertex $v\in P_B$.
It means that


$$(-1)^{n}F^{\ast}(P_B)=\sum_{T\in
T(B)}\widehat{F}(T^{\mathrm{op}})=\sum_{v\in
P_B}\sum_{f\in\overline{\sigma_v^{\mathrm{op}}}}\mathbf{x}_f.$$ It
remains to note that $f\in\overline{\sigma_v^{\mathrm{op}}}$ if
and only if $f$ is minimized at the vertex $v$.
\end{proof}


For $F\in QSym$ let $\chi(F,m)=\mbox{ps}_m(F)$ be the principal
specialization defined by algebraic extension of
$\mbox{ps}_m(x_i)=1$ for $1\leq i\leq m$ and $\mbox{ps}_m(x_i)=0$
for $i>m$. Since $\mbox{ps}_m(M_\alpha)={m\choose k(\alpha)}$ we
have
\[\chi(P_B,m)=\sum_{\alpha\models n}\zeta_\alpha(B){m \choose
k(\alpha)},\]which counts the number of $P_B$-generic functions
$f:[n]\rightarrow[m]$. It is related with
$\chi^\ast(P_B,m)=\mbox{ps}_m( F^\ast(P_B))$ by

\[\chi(P_B,-m)=(-1)^n\chi^\ast(P_B,m).\] Specially, for $m=1$,
we obtain the following

\begin{proposition}\label{vert}
The number of vertices $f_0(P_B)$ of a nestohedron $P_B$ is
determined by
 $f_0(P_B)=(-1)^{n}\chi(P_B,-1)$.
\end{proposition}
\begin{proof}
The statement follows from the identity
$\mbox{ps}_1(F^{\ast}(P_B))=c_{(n)}$, where $c_{(n)}$ is the
coefficient by $M_{(n)}$ in the expansion of $F^{\ast}(P_B)$ in
the monomial basis. By Proposition \ref{fstar} this coefficient
counts the vertices of $P_B$ that minimize $\langle f,x\rangle$
over $P_B$ for $f=(1,\ldots,1)$. But $f$ is orthogonal to $P_B$.
\end{proof}

\section{The graph invariant $F(P_{B(\Gamma)})$}

In this section we investigate the quasisymmetric function
$F(P_{B(\Gamma)})$ associated to a simple graph $\Gamma$.

For a graph $\Gamma$ on the vertex set $[n]$ and a subset
$I\subset [n]$ are defined the restriction $\Gamma\mid_I$ and the
contraction $\Gamma/I$. The restriction $\Gamma\mid_I$ is the
induced subgraph on the vertex set $I$ and the contraction
$\Gamma/I$ is a graph on $[n]\setminus I$ with two vertices $u$
and $v$ connected by the edge if either $\{u,v\}$ is an edge of
$\Gamma$ or there is a path $u,w_1,\ldots,w_k,v$ in $\Gamma$ with
$w_1,\ldots,w_k\in I$. To a graph $\Gamma$ is associated the
graphical building set $B(\Gamma)$ as in Example \ref{graph}. It
is immediate that $B(\Gamma\mid_I)=B(\Gamma)\mid_I$ and
$B(\Gamma/I)=B(\Gamma)/I$. An element $I\in B(\Gamma)$ and the
corresponding contraction $\Gamma/I$ are called in \cite{CD} the
tube and the reconnected complement respectively.

The vector space $\mathcal{G}$ spanned by all isomorphism classes
of simple graphs is endowed with the Hopf algebra structure by
operations

\[\Gamma_1\cdot\Gamma_2=\Gamma_1\sqcup\Gamma_2 \ \mbox{and} \
\Delta(\Gamma)=\sum_{I\subset V}\Gamma\mid_I\otimes\Gamma/I.\] The
map that associates the graphical building set $B(\Gamma)$ to a
graph $\Gamma$ is extended to a Hopf algebra monomorphism
$i:\mathcal{G}\rightarrow\mathcal{B}$. The induced character is
defined by $\zeta(\Gamma)=1$ if $\Gamma$ is discrete and zero
otherwise. It follows from Corollary \ref{main} that the
quasisymmetric function $F(P_{B(\Gamma)})$ is a multiplicative
graph invariant. By Theorem \ref{universal} it may be defined
purely in a graph theoretic manner.

Let $\Gamma$ be a simple graph on $n$ vertices
$V=\{v_1,\cdots,v_n\}$ and $\lambda:V\rightarrow\mathbb{N}$ be a
coloring with the set of colors $\{i_1<\cdots<i_k\}$. Define a
flag $\emptyset=I_0\subset I_1\subset\cdots\subset I_{k-1}\subset
I_k=V$ by $I_j=\lambda^{-1}(\{i_1,\cdots,i_j\})$ for $1\leq j\leq
k$. We say that $\lambda$ is a {\it ordered coloring} of $\Gamma$
if the graphs $\Gamma|_{I_j}/I_{j-1}$ are discrete for all $1\leq
j\leq k$. This means that each monochromatic set of vertices is
discrete and no two vertices of the same color are connected by a
path trough vertices colored by smaller colors. The type of an
ordered coloring $\lambda$ is the composition
$\mathrm{co}(\lambda)=(a_1,\cdots,a_k)\models n$, where
$a_j=|I_j\setminus I_{j-1}|$ is the number of vertices colored by
$i_j$, for all $1\leq j\leq k$. Let $Col^{\leq}(\Gamma)$ be the
set of all ordered colorings of the graph $\Gamma$ and $F_\Gamma$
be the enumerator function

\[F_\Gamma=\sum_{\lambda\in Col^{\leq}(\Gamma)}\mathbf{x}_\lambda.\]
By Theorem \ref{universal} it coincides with the quasisymmetric
function of the graph-associa\-hedra $B(\Gamma)$
$$F_\Gamma=F(P_{B(\Gamma)}).$$ Thus in the monomial basis it has the expansion
$F_\Gamma=\sum_{\alpha\models n}\zeta_\alpha(\Gamma)M_\alpha$,
where $\zeta_\alpha(\Gamma)$ is the number of ordered colorings
$\lambda:V\rightarrow\{1,\cdots,k(\alpha)\}$ of the type
$\mathrm{co}(\lambda)=\alpha$. The polynomial
$\chi(\Gamma,m)=\chi(B(\Gamma),m)$ counts the number of ordered
colorings with at most $m$ colors.

\begin{remark}

Stanley's chromatic symmetric function of a graph $X_\Gamma$
introduced in \cite{S1} is the enumerator of proper colorings
$\lambda:V(\Gamma)\rightarrow\mathbb{N}$. A coloring $\lambda$ is
proper if the graph $\Gamma$ does not contain a monochromatic
edge, i.e. the induced graph on $\lambda^{-1}(\{i\})$ for each
color $i\in\mathbb{N}$ is discrete. The sizes of monochromatic
parts define the type of the proper coloring which is a partition
of the number of vertices of the graph since ordering of colors is
inessential. The assignment $X_\Gamma$ is the canonical morphism
from the chromatic Hopf algebra of graphs to symmetric functions,
see (\cite{ABS}, Example 4.5). The coefficients $c_\mu(\Gamma)$ in
the expansion in the monomial basis of symmetric functions

$$X_\Gamma=\sum_{\mu\vdash n}c_\mu(\Gamma)m_\mu,$$
count the numbers of proper colorings of prescribed types
$\mu\vdash n$. Recall that $m_\mu=\sum_{s(\alpha)=\mu}M_\alpha$,
where the sum is over all compositions $\alpha\models n$ that can
be rearranged to the partition $\mu\vdash n$.

\end{remark}

The coefficients $\zeta_\alpha(\Gamma), \alpha\models n$ satisfy
the following properties.


\begin{proposition}\label{coeff}

Given a graph $\Gamma$ on the set of vertices $V=[n]$.

\begin{itemize}

\item[(a)]The coefficients $\zeta_{(k,1^{n-k})}(\Gamma), 1\leq
k\leq n$ determine the $f-$vector of the independence complex
$Ind(\Gamma)$ of the graph $\Gamma$.



\item[(b)]For any pair $\alpha\preceq\beta$ it holds
$\zeta_\alpha(\Gamma)\leq\zeta_\beta(\Gamma)$.

\item[(c)]$\zeta_\alpha(\Gamma)\leq c_{\mu}(\Gamma)$ for each
composition $\alpha\models n$ such that $s(\alpha)=\mu\vdash n$
and $c_\mu(\Gamma)$ are the coefficients of $X_\Gamma$ in the
monomial basis $\{m_\mu\}_{\mu\vdash n}$ of symmetric functions.
\end{itemize}

\end{proposition}

\begin{proof}
Recall that the coefficient $\zeta_\alpha(\Gamma)$ counts the
number of ordered colorings $\lambda:V\rightarrow[k(\alpha)]$ of
the type $\alpha\models n$.

\begin{itemize}
\item[(a)]The only condition for a coloring
$\lambda:V\rightarrow[n-k+1]$ to be ordered with ${\rm
type}(\lambda)=(k,1^{n-k})$ is that the set of vertices colored by
$1$ is $k$-element and discrete. Hence
$\zeta_{(k,1^{n-k})}(\Gamma)=(n-k)!f_{k-1}(Ind(\Gamma))$.



\item[(b)]Suppose that $\alpha$ is obtained from $\beta$ by
combining some of its adjacent parts, i.e.
$\alpha=(a_1,\ldots,a_i,\ldots,a_k)$ and
$\beta=(a_1,\ldots,a_i^{'},a_i^{''},\ldots,a_k)$ with
$a_i=a_i^{'}+a_i^{''}$. Then any ordered coloring of the type
$\alpha$ defines at least ${a_i \choose a_i^{'}}$ ordered
colorings of the type $\beta$.

\item[(c)]It is obvious since any ordered coloring of a type
$\alpha\models n$ is the coloring of the type $s(\alpha)\vdash n$.
\end{itemize}
\end{proof}



\begin{remark}
The formula similar to Proposition \ref{coeff} $(a)$ for the
$f$-vectors of simplicial complexes is derived in \cite[Section
4.5]{BHM}.
\end{remark}

The following theorem allows one to define the invariant
$F_\Gamma$ recursively starting with $F_{\emptyset}=M_{()}=1$.
Recall that $F\mapsto(F)_1$ is the shifting operator, see
Definition \ref{shift}.

\begin{theorem}\label{recursive}
For a connected graph $\Gamma$ on the vertex set $[n]$ it holds

\[F_\Gamma=\sum_{i\in [n]}(F_{\Gamma\setminus\{i\}})_1.\]

\end{theorem}
\begin{proof}
We arrange the vertices $v\in P_{B(\Gamma)}$ according to the
maximal elements of corresponding posets $P_v$. Let
$T(B(\Gamma))_i=\{T_v^{\circ}\mid v\in P_{B(\Gamma)}, \max
P_v=i\}$ be the multiset of specified unlabeled $B(\Gamma)$-trees.
Then by Theorem \ref{ppart} we have

\[F_\Gamma=\sum_{i\in[n]}\sum_{T\in T(B(\Gamma))_i}F(T).\] Denote by
$v_T$ the root of a tree $T$. The formula follows from the
recurrence formula for $T$-partitions enumerators, see Theorem
\ref{tree}

\[\sum_{T\in T(B(\Gamma))_i}F(T)=\sum_{T\in T(B(\Gamma))_i}(F(T\setminus\{v_T\}))_1=(F_{\Gamma\setminus\{i\}})_1.\]

\end{proof}

Theorem \ref{recursive} provides an effective computational tool
for the invariant $F_\Gamma$.

\begin{example}
The invariant $F_\Gamma$ distinguishes five-vertex graphs. In
particular, the unique pair of five-vertex graphs with the same
chromatic symmetric functions $X_\Gamma$ given in \cite[Figure
1]{S1} are distinguished by $F_\Gamma$. Figure \ref{same}
\footnote{The author is thankful to Marko Pe\v{s}ovi\'c for this
example who discovered it by using MathLab program.} shows a pair
of non-isomorphic six-vertex graphs with
$F_{\Gamma_1}=F_{\Gamma_2}=24M_{(1,2,1,1,1)}+96M_{(2,1,1,1,1)}+720M_{(1,1,1,1,1,1)}$.
On the other hand the chromatic numbers of graphs $\Gamma_1$ and
$\Gamma_2$ are different $\chi(\Gamma_1)=3$ and
$\chi(\Gamma_2)=4$. Since the chromatic number $\chi(\Gamma)$ can
be derived from Stanley's chromatic symmetric function $X_\Gamma$
we conclude that graph invariants $X_\Gamma$ and $F_\Gamma$ are
not comparable.

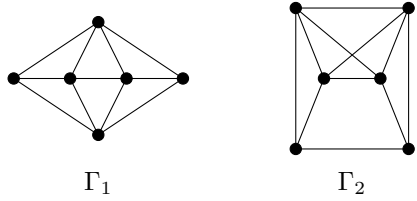
\begin{figure}[h!h!]\centering
\begin{tikzpicture}[scale=.75]
\draw
(-1.5,0)--(-.5,0)--(.5,0)--(1.5,0)--(0,1)--(-1.5,0)--(0,-1)--(1.5,0);
\draw (0,1)--(-.5,0)--(0,-1)--(.5,0)--(0,1); \draw [fill] (-1.5,0)
circle(.1cm);\draw [fill] (-.5,0) circle(.1cm);\draw [fill] (.5,0)
circle(.1cm);\draw [fill] (1.5,0) circle(.1cm);\draw [fill] (0,-1)
circle(.1cm);\draw [fill] (0,1) circle(.1cm);\node[below] at
(0,-1.5) {$\Gamma_1$};

\draw
(3.5,1.25)--(5.5,1.25)--(5.5,-1.25)--(3.5,-1.25)--(3.5,1.25)--(4,0)--(5,0)--(5.5,1.25)--(4,0);
\draw (3.5,1.25)--(5,0)--(5.5,-1.25);\draw (3.5,-1.25)--(4,0);
\draw [fill] (3.5,-1.25) circle(.1cm);\draw [fill] (3.5,1.25)
circle(.1cm);\draw [fill] (5.5,1.25) circle(.1cm);\draw [fill]
(5.5,-1.25) circle(.1cm);\draw [fill] (4,0) circle(.1cm);\draw
[fill] (5,0) circle(.1cm);\node[below] at (4.5,-1.5) {$\Gamma_2$};

\end{tikzpicture}
\caption{Graphs with $F_{\Gamma_1}=F_{\Gamma_2}$} \label{same}
\end{figure}

\noindent Note that $\Gamma_1=\overline{L_2L_4}$ and
$\Gamma_2=\overline{L_3L_3}$, where $\overline{\Gamma}$ denotes
the complement graph of $\Gamma$.
\end{example}

We obtain the recurrence relations satisfied by enumerators $F(Q)$
for $Q=Pe^{n-1},$ $As^{n-1}, Cy^{n-1}, St^{n-1}.$ We assume the
realization of $Q$ as a graph-associahedron of the corresponding
graph as in Example \ref{graph}. By convention the only
$(-1)$-dimensional polytope is $\emptyset$.

\begin{corollary}\label{correcursive}

For $n\geq 1$ the following recurrence relations hold

\[F(Pe^{n-1})=n(F(Pe^{n-2}))_1,\]
\[F(As^{n-1})=(\sum_{k=1}^{n}F(As^{k-2})F(As^{n-k-1}))_1,\]
\[F(Cy^{n-1})=n(F(As^{n-2}))_1,\]
\[F(St^{n-1})=((n-1)F(St^{n-2})+M_{(1)}^{n-1})_1.\]

\end{corollary}

From Proposition \ref{vert} and Corollary \ref{correcursive} we
recover the recurrence relations satisfied by numbers of vertices
of corresponding graph-associahedra. Note the identity
$\chi((F)_1,-1)=-\chi(F,-1)$ which is a consequence of
$\chi(M_\alpha,-1)=(-1)^{k(\alpha)}$.

\begin{corollary}
For $n\geq 1$ we have that the numbers $p_n=f_0(Pe^{n-1})$, \
$a_n=f_0(As^{n-1}), c_n=f_0(Cy^{n-1})$ and $s_n=f_0(St^{n-1})$
satisfy

\[p_n=np_{n-1},
a_n=\sum_{k=1}^{n}a_{k-1}a_{n-k}, c_n=na_{n-1},
s_n=(n-1)s_{n-1}+1\] with $p_1=a_1=c_1=s_1=1$. Therefore $p_n=n!,
a_n=\frac{1}{n+1}{2n \choose n}, c_n={2n-2 \choose n-1}$ and
$s_n=(n-1)!\sum_{k=0}^{n-1}\frac{1}{k!}$.
\end{corollary}

\section{Conclusion}

We conclude with several natural questions in connection with the
Hopf algebra $\mathcal{B}$ and the graph invariant $F_\Gamma$.

\begin{problem}
In a combinatorial Hopf algebra are defined the generalized
Dehn-Sommerville relations which characterize the odd subalgebra
(see \cite{ABS}, Section 5). Find a graph or a building set that
satisfies the generalized Dehn-Sommerville relations for
$\mathcal{B}$. The same problem is resolved in \cite{GSJ} for the
chromatic Hopf algebra of hypergraphs, where the whole class of
solutions called eulerian hypergraphs are found.
\end{problem}

\begin{problem}
Stanley asked whether the chromatic symmetric function $X_\Gamma$
is a complete invariant of trees. This question is still opened.
The same question is natural to be posed for the enumerator
$F_\Gamma$.
\end{problem}






\end{document}